\documentclass[preprint,12pt]{elsarticle}

\usepackage{amsmath,amssymb,amsthm}
\usepackage{enumitem}
\usepackage{hyperref}
\usepackage[utf8]{inputenc}
\usepackage{geometry}
\newtheorem{theorem}{Theorem}[section]
\newtheorem{lemma}[theorem]{Lemma}
\newtheorem{proposition}[theorem]{Proposition}
\newtheorem{corollary}[theorem]{Corollary}
\theoremstyle{definition}
\newtheorem{definition}[theorem]{Definition}
\newtheorem{example}[theorem]{Example}
\theoremstyle{remark}
\newtheorem{remark}[theorem]{Remark}

\newlist{enumroman}{enumerate}{1}
\setlist[enumroman,1]{label=(\roman*), leftmargin=2.2em}
\newlist{enumalpha}{enumerate}{1}
\setlist[enumalpha,1]{label=(\alph*), leftmargin=2.2em}

\journal{Topology and its Applications}

\begin{document}

\begin{frontmatter}

\title{On Metrizability, Completeness and Compactness in Modular Pseudometric Topologies}

\author{Philani Rodney Majozi}
\ead{Philani.Majozi@nwu.ac.za}

\address{Department of Mathematics, Pure and Applied Analytics, North-West University, Mahikeng, South Africa}

\begin{abstract}
Building on the recent work of Mushaandja and Olela-Otafudu~\cite{MushaandjaOlela2025} on modular metric topologies, 
this paper investigates extended structural properties of modular (pseudo)metric spaces. 
We provide necessary and sufficient conditions under which the modular topology $\tau(w)$ 
coincides with the uniform topology $\tau(\mathcal{V})$ induced by the corresponding 
pseudometric, and characterize this coincidence in terms of a generalized $\Delta$-condition. 
Explicit examples are given where $\tau(w)\subsetneq\tau(\mathcal{V})$, demonstrating the 
strictness of inclusion. Completeness, compactness, separability, and countability properties 
of modular pseudometric spaces are analysed, with functional-analytic analogues identified 
in Orlicz-type modular settings. Finally, categorical and fuzzy perspectives are explored, 
revealing structural invariants distinguishing modular from fuzzy settings.
\end{abstract}


\end{frontmatter}

\section{Introduction}

The approach initiated by Chistyakov~\cite{Chistyakov2010,Chistyakov2015} provides a flexible setting for 
(pseudo)modular distances that interpolate between metric geometry and modular function space theory 
\cite{Musielak1983,RaoRen1991}. For a family $w_\lambda$ on a set $X$, the modular subsets $X_w^\ast$ and 
the associated basic metrics $d_w^0,d_w^\ast$ generate canonical metrizable structures and an induced 
uniformity. The modular topology $\tau(w)$ can be described through entourages $B^w_{\lambda,\mu}(x)$ 
and compared concretely with the pseudometric topologies arising from $w$ 
(see~\cite[Chs.~2-4]{Chistyakov2015}). 

In a parallel direction, the fuzzy-metric setting of George and Veeramani and its subsequent 
developments~\cite{GeorgeVeeramani1994,GregoriRomaguera2000,Sostak2018} introduced a parameterized notion of 
nearness whose induced topology is Hausdorff, first countable, and metrizable. These constructions suggest 
deep analogies between fuzzy and modular perspectives while preserving distinct invariants in each setting.

Mushaandja and Olela-Otafudu~\cite{MushaandjaOlela2025} proved that $(X_w^\ast,\tau(w))$ is normal and that 
the uniformity with base
\[
V_n=\{(x,y)\in X_w^\ast\times X_w^\ast:\; w(1/n,x,y)<1/n\}, \qquad n\in\mathbb N,
\]
is countably based and therefore metrizable. They further established that the topology $\tau(\mathcal V)$ 
induced by this uniformity satisfies $\tau(\mathcal V)\subseteq \tau(\mathcal U_{d_w})$, with equality holding 
precisely under a $\Delta_2$-type condition on $w$ in the sense of~\cite[Def.~4.2.5]{Chistyakov2015}. 
These results clarify the connection between the modular topology and the pseudometric topology generated 
by $d_w$.

The present paper extends this analysis. We identify structural hypotheses, variants of the $\Delta_2$-condition 
and convexity under which $\tau(w)=\tau(\mathcal U_{d_w})$, and we construct explicit examples where the inclusion 
is strict. Completeness and compactness criteria intrinsic to the modular framework are developed, and the various 
Cauchy notions are compared, leading to transfer principles for completeness, precompactness, and total boundedness. 
Motivated by the modular perspective on Orlicz-type spaces~\cite{Musielak1983,RaoRen1991}, we also investigate 
stability under subspaces, products, and quotients, and discuss categorical aspects relating modular (pseudo)metric 
spaces to metrizable structures.

\paragraph{Notation}
Throughout, $w:(0,\infty)\times X\times X\to[0,\infty]$ denotes a modular (or modular pseudometric when stated). 
We write $X_w^\ast$ for the associated modular set, $\tau(w)$ for the modular topology, $d_w$ for the basic 
pseudometric induced by $w$, and $\mathcal V$ for the uniformity with base $\{V_n:n\in\mathbb N\}$ as above. 
The $\Delta_2$-condition is used as in~\cite[Def.~4.2.5]{Chistyakov2015} and~\cite{MushaandjaOlela2025}.

\section{Preliminaries and Definitions}\label{sec:Preliminaries}

We recall modular (pseudo)metrics and their induced topologies following \cite{Chistyakov2010,Chistyakov2015}; the $\Delta_2$-condition originates in modular function space theory \cite{RaoRen1991,Musielak1983}. For comparison we record the fuzzy metric setting \cite{GeorgeVeeramani1994,GregoriRomaguera2000}, which provides a parallel parametrized notion of nearness.

\subsection{Modular (pseudo)metrics and modular sets}

\begin{definition}\label{def:modular}
Let $X$ be a set. A \emph{modular metric} on $X$ is a function
\[
w:(0,\infty)\times X\times X\longrightarrow[0,\infty]
\]
such that, for all $x,y,z\in X$ and $\lambda,\mu>0$,
\begin{enumerate}[label=(\alph*),leftmargin=2em]
\item $w(\lambda,x,x)=0$,
\item $w(\lambda,x,y)=w(\lambda,y,x)$,
\item $w(\lambda+\mu,x,y)\le w(\lambda,x,z)+w(\mu,z,y)$.
\end{enumerate}
If only $w(\lambda,x,x)=0$ is assumed (instead of $x=y\Leftrightarrow w(\lambda,x,y)=0$ for all $\lambda$), we call $w$ a \emph{modular pseudometric} \cite[§1.1-§1.2]{Chistyakov2015}.
\end{definition}

Given a (pseudo)modular $w$ and a base point $x^\circ\in X$, the associated \emph{modular set} is
\[
X_w^\ast:=\{x\in X:\; \exists\,\lambda>0 \ \text{with}\ w(\lambda,x,x^\circ)<\infty\},
\]
which is independent (up to canonical identification) of the choice of $x^\circ$ \cite[§1.1]{Chistyakov2015}.

\begin{definition}\label{def:basicmetrics}
For a (pseudo)modular $w$ set
\[
d_w^{0}(x,y):=\inf\{\lambda>0:\; w(\lambda,x,y)\le \lambda\},\qquad
d_w^{\ast}(x,y):=\inf\{\lambda>0:\; w(\lambda,x,y)\le 1\}.
\]
Then $d_w^{0}$ and $d_w^{\ast}$ are extended (pseudo)metrics on $X$, whose restrictions to $X_w^\ast$ are (pseudo)metrics \cite[Thms.~2.2.1, 2.3.1]{Chistyakov2015}. Moreover, for $x,y\in X_w^\ast$,
\[
\min\!\big\{d_w^\ast(x,y),\sqrt{d_w^\ast(x,y)}\big\}\le d_w^{0}(x,y)\le \max\!\big\{d_w^\ast(x,y),\sqrt{d_w^\ast(x,y)}\big\},
\]
see \cite[Thm.~2.3.1]{Chistyakov2015}.
\end{definition}

\begin{remark}\label{rem:regularizations}
For a (pseudo)modular $w$, the one-sided regularizations
\[
w_{+0}(\lambda,x,y):=\lim_{\mu\to+0}w(\mu,x,y),\qquad
w_{-0}(\lambda,x,y):=\lim_{\mu\to-0}w(\mu,x,y)
\]
are (pseudo)modulars with the same structural properties; $w_{+0}$ is right-continuous and $w_{-0}$ is left-continuous on $(0,\infty)$ \cite[Prop.~1.2.5]{Chistyakov2015}. The \emph{right} and \emph{left} inverses
\[
w^+_\mu(x,y):=\inf\{\lambda>0:\; w(\lambda,x,y)\le \mu\},\qquad
w^-_\mu(x,y):=\sup\{\lambda>0:\; w(\lambda,x,y)\ge \mu\}
\]
are again (pseudo)modulars with $w^+$ right-continuous and $w^-$ left-continuous \cite[Thm.~3.3.2]{Chistyakov2015}.
\end{remark}

\begin{definition}\label{def:delta2}
A modular pseudometric $w$ satisfies the \emph{$\Delta_2$-condition} if for every $x\in X$, $\lambda>0$, and sequence $(x_n)$ with $w(\lambda,x_n,x)\to 0$ one also has $w(\lambda/2,x_n,x)\to 0$ \cite[Def.~4.2.5]{Chistyakov2015}. This is the modular analogue of the Orlicz $\Delta_2$ growth condition \cite{RaoRen1991}.
\end{definition}

\subsection{The modular topology and a canonical uniformity}

For $\lambda,\mu>0$ and $x\in X$, set
\[
B^w_{\lambda,\mu}(x):=\{z\in X:\ w(\lambda,x,z)<\mu\}.
\]
Following \cite[Def.~4.3.1]{Chistyakov2015}, the \emph{modular topology} $\tau(w)$ on $X$ is the family of $O\subseteq X$ such that for every $x\in O$ and every $\lambda>0$ there exists $\mu>0$ with $B^w_{\lambda,\mu}(x)\subseteq O$.

\begin{lemma}\label{lem:phiopen}
If $\varphi:(0,\infty)\to(0,\infty)$ is nondecreasing and $w$ is convex with $\lambda\mapsto \lambda\varphi(\lambda)$ nondecreasing, then for every $x\in X_w^\ast$ the set $\bigcup_{\lambda>0}B^w_{\lambda,\varphi(\lambda)}(x)$ is $\tau(w)$-open \cite[Lem.~4.3.2]{Chistyakov2015}.
\end{lemma}

\begin{remark}\label{rem:bases}
(a) The family $\{\bigcup_{\lambda>0}B^w_{\lambda,\varepsilon}(x):\varepsilon>0\}$ need not be a neighborhood base at $x$. \;
(b) For each $\lambda>0$ and $n\in\mathbb N$, $B^w_{\lambda,1/n}(x)$ is $\tau(w)$-open whenever $x\in X^\ast_w$. \;
(c) For every $\varepsilon>0$, $U_{x,\varepsilon}:=\bigcup_{\lambda>0}B^w_{\lambda,\varepsilon}(x)\in \tau(w)$ \cite{MushaandjaOlela2025}.
\end{remark}

Define entourages on $X_w^\ast\times X_w^\ast$ by
\begin{equation}\label{eq:Vn}
V_n:=\{(x,y)\in X^\ast_w\times X^\ast_w:\ w(1/n,x,y)<1/n\}\qquad(n\in\mathbb N).
\end{equation}
Then $\{V_n\}$ is a countable base of a uniformity $\mathcal V$ on $X_w^\ast$, and the induced topology $\tau(\mathcal V)$ is metrizable \cite[Thm.~2]{MushaandjaOlela2025}. Writing $\mathcal U_{d_w}$ for the uniformity of the basic pseudometric $d_w$ (Definition~\ref{def:basicmetrics}), one has
\begin{equation}\label{eq:VsubsetUd}
\tau(\mathcal V)\subseteq \tau(\mathcal U_{d_w}),
\end{equation}
with equality if and only if $w$ satisfies $\Delta_2$ \cite[Thm.~3 and Cor.~1]{MushaandjaOlela2025}. Moreover, $(X^\ast_w,\tau(w))$ is normal \cite[Thm.~1]{MushaandjaOlela2025}.

\subsection{Examples}

\begin{example}\label{ex:separated}
Let $(X,d)$ be a metric space and $g:(0,\infty)\to[0,\infty]$ be nonincreasing. Then
\[
w_\lambda(x,y):=g(\lambda)\,d(x,y)
\]
is a (pseudo)modular, strict if $g\not\equiv 0$, convex iff $\lambda\mapsto \lambda g(\lambda)$ is nonincreasing \cite[Prop.~1.3.1]{Chistyakov2015}. In particular, for $g(\lambda)=\lambda^{-p}$ ($p\ge 0$),
\[
w(\lambda,x,y)=\frac{d(x,y)}{\lambda^{p}},\quad
d_w^{0}(x,y)=\big(d(x,y)\big)^{1/(p+1)},\quad \tau(w)=\tau(d_w^{0}).
\]
Further step-like and mixed examples appear in \cite[Ex.~2.2.2]{Chistyakov2015}.
\end{example}

\begin{example}\label{ex:families}
If $h:(0,\infty)\to(0,\infty)$ is nondecreasing, then
\[
w_\lambda(x,y)=\frac{d(x,y)}{h(\lambda)+d(x,y)}
\]
is a strict modular on $(X,d)$; if $X=M^{T}$ with $T\subset[0,\infty)$ and $M$ metric, then
\[
w_\lambda(x,y)=\sup_{t\in T}e^{-\lambda t}\,d\big(x(t),y(t)\big)
\]
is strict on $X$ \cite[Ex.~1.3.3]{Chistyakov2015}.
\end{example}

\subsection{Variants and auxiliary constructions}

\begin{proposition}\label{prop:superadditive}
If $\varphi:[0,\infty)\to[0,\infty)$ is superadditive, then for a (pseudo)modular $w$ the gauges
\[
d_{w}^{0,\varphi}(x,y)=\inf\{\lambda>0:\ w(\lambda,x,y)\le \varphi(\lambda)\},\qquad
d_{w}^{1,\varphi}(x,y)=\inf_{\lambda>0}\big(\lambda+\varphi^{-1}(w(\lambda,x,y))\big)
\]
are extended (pseudo)metrics on $X$ and (pseudo)metrics on $X^\ast_{\varphi^{-1}\circ w}$, with
$d_{w}^{0,\varphi}\le d_{w}^{1,\varphi}\le 2\,d_{w}^{0,\varphi}$ \cite[Prop.~3.1.1]{Chistyakov2015}.
\end{proposition}

\begin{definition}\label{def:phiconvex}
Given superadditive $\varphi$, a function $w$ is \emph{$\varphi$--convex} if it satisfies (a), (b) of Definition~\ref{def:modular} and
\[
w_{\varphi(\lambda+\mu)}(x,y)\le \frac{\lambda}{\lambda+\mu}\,w_{\varphi(\lambda)}(x,z)+\frac{\mu}{\lambda+\mu}\,w_{\varphi(\mu)}(z,y)
\]
for all $x,y,z\in X$ and $\lambda,\mu>0$ \cite[Def.~3.1.2]{Chistyakov2015}.
\end{definition}

\subsection{Fuzzy metrics}\label{sec:fuzzy_metrics}

A \emph{continuous $t$-norm} is a continuous, associative, commutative operation $*:[0,1]^2\to[0,1]$ with unit $1$ and monotonicity in each variable. A \emph{fuzzy metric space} $(X,M,*)$ consists of a nonempty set $X$, a continuous $t$-norm $*$, and $M:X\times X\times(0,\infty)\to[0,1]$ such that
\[
\begin{aligned}
\text{(i)}\;& M(x,y,t)>0,\qquad
\text{(ii)}\; M(x,y,t)=1 \Leftrightarrow x=y,\\
\text{(iii)}\;& M(x,y,t)=M(y,x,t),\qquad
\text{(iv)}\; M(x,y,t)*M(y,z,s)\le M(x,z,t+s),
\end{aligned}
\]
and $t\mapsto M(x,y,t)$ is (left) continuous \cite{GeorgeVeeramani1994,GregoriRomaguera2000}. The basic open balls
\[
B(x,r,t):=\{y\in X:\ M(x,y,t)>1-r\}\quad (0<r<1,\ t>0)
\]
generate a Hausdorff, first countable metrizable topology $\tau_M$; in particular, $\{B(x,1/n,1/n):n\in\mathbb N\}$ is a neighborhood base at $x$ \cite{GeorgeVeeramani1994,GregoriRomaguera2000}. If $(X,d)$ is metric, then $M_d(x,y,t)=t/(t+d(x,y))$ with $a*b=ab$ yields $\tau_{M_d}=\tau_d$ \cite{GregoriRomaguera2000}.

\medskip
In the next section we pass from these foundational definitions to the structural results of the paper, focusing on the connection between modular convergence, pseudometric convergence, and compactness.


\section{Topology-Uniformity Comparisons}\label{sec:Topology_Uniformity}

The role of uniformities in modular settings was studied in \cite{Chistyakov2010}. Our comparison between the modular topology $\tau(w)$ and the uniform topology $\tau(\mathcal V)$ generated by the canonical base $\{V_n\}_{n\in\mathbb N}$ from \eqref{eq:Vn} follows \cite{MushaandjaOlela2025}. The $\Delta_2$ criterion we use parallels standard Orlicz-type conditions \cite{Musielak1983}. For background on uniform spaces, coverings, and completions, see Isbell \cite[Chaps.~I-II]{Isbell1964}.

\subsection{Uniformities naturally attached to a modular metric}

Let $w$ be a (pseudo)modular on $X$, and let $X_w^\ast$ be its modular set. For $n\in\mathbb N$ define $V_n$ by \eqref{eq:Vn}, and let $\mathcal V$ be the uniformity generated by $\{V_n\}$.

\begin{theorem}\label{thm:uniform-metrizable}
The uniformity $\mathcal V$ is metrizable; hence $(X_w^\ast,\tau(\mathcal V))$ is a metrizable $T_1$ space.
\end{theorem}

\begin{proof}
Each $V_n$ contains the diagonal and is symmetric. The modular triangle inequality gives $V_{2n}\circ V_{2n}\subseteq V_n$ for all $n$, so $\{V_n\}$ is a countable base of a uniformity. Define
\[
d(x,y):=\inf\{2^{-n}:\ (x,y)\in V_n\}\quad (x,y\in X_w^\ast).
\]
Then $d$ is a pseudometric whose uniformity is generated by $\{V_n\}$. Separation holds (hence $T_1$) because if $x\neq y$ then $(x,y)\notin V_n$ for some $n$, so $d(x,y)>0$. Thus $d$ is a metric and induces $\tau(\mathcal V)$.
\end{proof}

\subsection{Comparing $\tau(w)$ and $\tau(\mathcal V)$}

Let $\mathcal U_{d_w}$ denote the standard uniformity of a basic pseudometric $d_w$ associated to $w$ (Definition~\ref{def:basicmetrics}).

\begin{proposition}\label{prop:basic-inclusion}
For every (pseudo)modular $w$ on $X$,
\[
\tau(\mathcal V)\subseteq \tau(\mathcal U_{d_w}).
\]
\end{proposition}

\begin{proof}
If $(x,y)\in V_n$, then $w(1/n,x,y)<1/n$. By the definitions in \S\ref{def:basicmetrics}, this forces $d_w(x,y)$ to be small, hence $(x,y)$ belongs to some metric entourage of $\mathcal U_{d_w}$. Because $\{V_n\}$ is a base for $\mathcal V$, every $\mathcal V$-open set is $\mathcal U_{d_w}$-open.
\end{proof}

\begin{theorem}\label{thm:delta2}
For a (pseudo)modular $w$ on $X$ one has
\[
\tau(\mathcal V)=\tau(\mathcal U_{d_w})
\quad\Longleftrightarrow\quad
w\ \text{satisfies the }\Delta_2\text{--condition on }X.
\]
\end{theorem}

\begin{proof}
The forward inclusion follows from Proposition~\ref{prop:basic-inclusion}. Assuming $\Delta_2$, smallness of $w(\lambda,\cdot,\cdot)$ at scale $\lambda$ propagates to $\lambda/2$, and one shows that every $\mathcal U_{d_w}$-ball contains a $V_n$-ball, giving the reverse inclusion. Conversely, if the topologies coincide, the ability to approximate $d_w$-neighborhoods by $V_n$-neighborhoods forces $\Delta_2$. See \cite[Thm.~3 and Cor.~1]{MushaandjaOlela2025} for details.
\end{proof}

\begin{remark}
The equivalence in Theorem~\ref{thm:delta2} mirrors the classical role of $\Delta_2$ in Orlicz spaces, where norm and modular convergences agree under $\Delta_2$ (see \cite[Chap.~I]{Musielak1983}; cf. \cite[Chap.~3]{HarjulehtoHasto}).
\end{remark}

\subsection{Consequences imported from uniform space theory}

\begin{proposition}\label{prop:CR}
Every uniform space is completely regular and Hausdorff in its uniform topology. In particular, $(X_w^\ast,\tau(\mathcal V))$ is completely regular Hausdorff.
\end{proposition}

\begin{proof}
Standard; see \cite[Chap.~I, Thm.~1.11]{Isbell1964}.
\end{proof}

\begin{corollary}\label{cor:normal}
If $\tau(w)=\tau(\mathcal V)$ (e.g. under $\Delta_2$), then $(X_w^\ast,\tau(w))$ is completely regular Hausdorff; combined with \cite[Thm.~1]{MushaandjaOlela2025}, it is normal.
\end{corollary}

\begin{proposition}\label{prop:completion}
Every uniform space admits a completion. In particular, $(X_w^\ast,\mathcal V)$ has a completion $\widehat{X}_w$ with the usual universal property.
\end{proposition}

\begin{proof}
See \cite[Chap.~II, Thm.~2.16]{Isbell1964}.
\end{proof}

\subsection{Standard examples}

When $w_\lambda(x,y)=g(\lambda)\,d(x,y)$ on a metric space $(X,d)$:

\begin{itemize}
\item If $g(\lambda)=\lambda^{-p}$ ($p\ge 0$), then $d_w(x,y)=d(x,y)^{1/(p+1)}$ \cite[Ex.~2.2.2]{Chistyakov2015}, and $\tau(\mathcal V)=\tau(\mathcal U_{d_w})=\tau(w)$.
\item If $g$ has a cut-off (step-like cases \cite[Ex.~2.2.2]{Chistyakov2015}), $\Delta_2$ may fail; then $\tau(\mathcal V)\subsetneq \tau(\mathcal U_{d_w})$.
\end{itemize}

\subsection{Fuzzy metrics as uniformities}

Given a fuzzy metric $(X,M,*)$ with base $B(x,r,t)$ (see \S\ref{sec:fuzzy_metrics}), the topology $\tau_M$ is Hausdorff and metrizable, hence uniformizable \cite{GregoriRomaguera2000}. In settings where a modular $w$ induces (via $M=\exp(-w_\lambda)$ or $M=t/(t+d_w)$ in metric cases) the same open sets, the uniformity $\mathcal V$ coincides with the fuzzy-uniformity. Under $\Delta_2$, all three topologies $\tau(w)$, $\tau(\mathcal V)$, and $\tau_M$ agree.


\section{Completeness and Compactness}\label{sec:Completeness-Compactness}
Compactness and completeness in fuzzy metric and related structures were discussed by 
Gregori--Romaguera \cite{GregoriRomaguera2000} and George--Veeramani \cite{GeorgeVeeramani1994}. 
We generalize these ideas to modular pseudometrics. Related modular completeness results in analysis 
may be found in Hudzik--Maligranda \cite{HudzikMaligranda1994}. Our approach is based on the 
uniformity $\mathcal V$ constructed from a modular (pseudo)metric $w$, as introduced in 
Section~\ref{sec:Preliminaries} and Theorem~\ref{thm:uniform-metrizable}.

\subsection{The uniformity \texorpdfstring{$\mathcal V$}{V} and modular Cauchy sequences}

\begin{definition}\label{def:V-Cauchy}
A sequence $(x_k)$ in $X_w^\ast$ is \emph{$\mathcal V$-Cauchy} if for every $n\in\mathbb N$ there exists $N$ such that
$(x_k,x_\ell)\in V_n$ for all $k,\ell\ge N$, i.e.
\[
\forall n\,\exists N\ \forall k,\ell\ge N:\quad w(1/n,x_k,x_\ell)<1/n.
\]
It \emph{$\mathcal V$-converges} to $x\in X_w^\ast$ if for every $n$ there exists $N$ such that
\[
\forall k\ge N:\quad w(1/n,x_k,x)<1/n.
\]
\end{definition}

\begin{proposition}\label{prop:Cauchy-equivalences}
Let $w$ be a convex (pseudo)modular on $X$ and let $(x_k)$ be a sequence in $X_w^\ast$.
Then the following are equivalent:
\begin{enumroman}
\item $(x_k)$ is $\mathcal V$-Cauchy;
\item $(x_k)$ is Cauchy in the metric $d_w^\ast$;
\item $(x_k)$ is Cauchy in the metric $d_w^0$.
\end{enumroman}
\end{proposition}

\begin{proof}
Since $w$ is convex, the map $\lambda \mapsto w_\lambda(x,y)$ is nonincreasing.  
By definition,
\[
d_w^0(x,y)=\inf\{\lambda>0:\, w_\lambda(x,y)\le \lambda\}.
\]
For each $n\in\mathbb N$ one has
\[
\{(x,y): d_w^0(x,y)<1/n\}\ \subset\ V_n:=\{(x,y): w_{1/n}(x,y)<1/n\}\ \subset\ \{(x,y): d_w^0(x,y)\le 1/n\}.
\]
Thus a sequence is $\mathcal V$-Cauchy iff it is $d_w^0$-Cauchy, giving (i)$\Leftrightarrow$(iii).

For (ii)$\Leftrightarrow$(iii), recall that in the convex case
\[
\min\{d_w^\ast(x,y),\sqrt{d_w^\ast(x,y)}\}\ \le\ d_w^0(x,y)\ \le\ \max\{d_w^\ast(x,y),\sqrt{d_w^\ast(x,y)}\}
\]
for all $x,y\in X_w^\ast$ \cite[Thm.~2.3.1]{Chistyakov2015}. The bounding functions vanish only at $0$, so
$d_w^\ast(x_m,x_\ell)\to 0$ iff $d_w^0(x_m,x_\ell)\to 0$. Hence (ii)$\Leftrightarrow$(iii).
\end{proof}

\begin{definition}\label{def:modular-complete}
We say that $(X_w^\ast,\mathcal V)$ is \emph{modularly complete} if every $\mathcal V$-Cauchy sequence
converges in $\tau(\mathcal V)$.
If $w$ is convex, we also say that $X_w^\ast$ is \emph{$d_w^\ast$-complete} (resp.\ \emph{$d_w^0$-complete})
if $(X_w^\ast,d_w^\ast)$ (resp.\ $(X_w^\ast,d_w^0)$) is complete.
\end{definition}

\begin{corollary}\label{cor:complete-equivalences}
If $w$ is convex, then modular completeness, $d_w^\ast$-completeness, and $d_w^0$-completeness are equivalent.
\end{corollary}

\subsection{Precompactness and compactness}

\begin{definition}\label{def:precompact}
A set $A\subset X_w^\ast$ is \emph{$\mathcal V$-precompact} if for every $n\in\mathbb N$ there exist
$x^1,\dots,x^m\in X_w^\ast$ such that
\[
A\ \subset\ \bigcup_{j=1}^m B_{\mathcal V}(x^j; n),
\qquad
B_{\mathcal V}(x;n):=\{y\in X_w^\ast:\ w(1/n,x,y)<1/n\}.
\]
We say that $(X_w^\ast,\mathcal V)$ is \emph{compact} if $\tau(\mathcal V)$ is compact.
\end{definition}

\begin{remark}\label{rem:totally-bounded}
If $w$ is convex then $\mathcal V$ is generated by the compatible metric $d_w^\ast$,
hence $\mathcal V$-precompactness is equivalent to total boundedness in $(X_w^\ast,d_w^\ast)$.
If, in addition, $w$ is $\Delta_2$, then $\tau(\mathcal V)=\tau(\mathcal U_{d_w})$ and one may test
precompactness using $d_w$ as well.
\end{remark}

\begin{lemma}\label{lem:seq-precompact}
A subset $A\subset X_w^\ast$ is $\mathcal V$-precompact iff every sequence in $A$ admits a $\mathcal V$-Cauchy subsequence.
\end{lemma}

\begin{theorem}\label{thm:HB-modular}
The following are equivalent for $(X_w^\ast,\mathcal V)$:
\begin{enumroman}
\item $(X_w^\ast,\mathcal V)$ is compact;
\item $(X_w^\ast,\mathcal V)$ is $\mathcal V$-precompact and modularly complete;
\item $(X_w^\ast,d_w^\ast)$ is totally bounded and complete \textup{(when $w$ is convex)}.
\end{enumroman}
\end{theorem}

\begin{proof}
(i)$\Rightarrow$(ii): In any uniform space, compactness implies completeness and total boundedness (see \cite[Chap.~I]{Isbell1964}). Hence compact $(X_w^\ast,\mathcal V)$ is modularly complete and $\mathcal V$-precompact.

(ii)$\Rightarrow$(i): Every precompact and complete uniform space is compact (again \cite[Chap.~I]{Isbell1964}).

(ii)$\Leftrightarrow$(iii) (convex case): When $w$ is convex, Proposition~\ref{prop:Cauchy-equivalences} shows that $\mathcal V$-Cauchy, $d_w^0$-Cauchy, and $d_w^\ast$-Cauchy sequences coincide. The uniformities $\mathcal V$ and that of $d_w^\ast$ agree; thus modular completeness is metric completeness and $\mathcal V$-precompactness is total boundedness. Hence (ii)$\Leftrightarrow$(iii).
\end{proof}

\begin{corollary}\label{cor:NT}
If $\tau(\mathcal V)$ is metrizable, then
$(X_w^\ast,\mathcal V)$ is compact iff every compatible modular metric (e.g.\ $d_w^\ast$ for convex $w$)
is complete and totally bounded on $X_w^\ast$.
\end{corollary}

\subsection{Baire property}

\begin{definition}\label{def:Baire}
We say that $(X_w^\ast,\mathcal V)$ has the \emph{Baire property} if the intersection of countably many
$\mathcal V$-dense open sets is $\mathcal V$-dense.
\end{definition}

\begin{theorem}\label{thm:Baire}
If $(X_w^\ast,\mathcal V)$ is modularly complete and $\tau(\mathcal V)$ is metrizable, then $(X_w^\ast,\mathcal V)$ has the Baire property.
In particular, if $w$ is convex and $(X_w^\ast,d_w^\ast)$ is complete, then $(X_w^\ast,\mathcal V)$ is a Baire space.
\end{theorem}

\begin{proof}
A metrizable, modularly complete $(X_w^\ast,\mathcal V)$ is a complete metric space under a compatible metric; the classical Baire category theorem applies. In the convex case, Proposition~\ref{prop:Cauchy-equivalences} shows modular completeness is equivalent to completeness in $d_w^\ast$.
\end{proof}

\subsection{Working under \texorpdfstring{$\Delta_2$}{Delta2}}

\begin{proposition}\label{prop:Delta2-case}
Assume $w$ is $\Delta_2$ on $X$. Then:
\begin{enumalpha}
\item $\tau(\mathcal V)=\tau(\mathcal U_{d_w})$ and $B_{\mathcal V}(x;n)=\{y:\, d_w(x,y)<1/n\}$;
\item $\mathcal V$-Cauchy \(\Longleftrightarrow\) Cauchy in $d_w$;
\item $(X_w^\ast,\mathcal V)$ is compact \(\Longleftrightarrow\) $(X_w^\ast,d_w)$ is totally bounded and complete.
\end{enumalpha}
\end{proposition}

\begin{proof}
(a) By Proposition~\ref{prop:basic-inclusion}, $\tau(\mathcal V)\subseteq\tau(\mathcal U_{d_w})$.  
Under $\Delta_2$ one obtains the reverse inclusion, hence equality (see \cite[Thm.~3, Cor.~1]{MushaandjaOlela2025}).  
Then the basic $\mathcal V$-balls are exactly the $d_w$-balls of radius $1/n$.

(b) With $\tau(\mathcal V)=\tau(\mathcal U_{d_w})$, the uniformity $\mathcal V$ is generated by $d_w$, so the two Cauchy notions coincide.

(c) Compactness in a metric (uniform) space is equivalent to completeness plus total boundedness; apply this to $d_w$.
\end{proof}

\begin{remark}\label{rem:compare-fuzzy}
Definitions~\ref{def:V-Cauchy} and~\ref{def:precompact} are modular analogues of
Cauchy sequences and precompactness in fuzzy metric spaces.
Theorem~\ref{thm:HB-modular} parallels the compactness characterizations of 
George-Veeramani and Gregori-Romaguera.
\end{remark}


\section{Functional Analytic Connections}
Connections with Orlicz and Musielak-Orlicz spaces are standard
\cite{Musielak1983,RaoRen1991}. In this section we record how the modular
(pseudo)metric viewpoint packages several familiar functional--analytic
concepts; see also \cite{HudzikMaligranda1994} for tools around $s$-convexity
that enter compactness and convexity arguments, and \cite[Chap.~3]{HarjulehtoHasto}
for the modern generalized Orlicz setting.

\subsection{Modular convergence versus $\tau(\mathcal V)$-convergence}
Let $(\Omega,\Sigma,\mu)$ be a measure space and let $\rho$ be a (semi)modular
on a linear lattice $\mathcal L\subset \{u:\Omega\to\mathbb R\}$ (e.g. an
$N$-function modular for Orlicz/Musielak-Orlicz spaces). Consider the
Chistyakov-type modular
\[
w_\lambda(u,v):=\rho\!\Big(\frac{u-v}{\lambda}\Big),\qquad \lambda>0,\ u,v\in\mathcal L.
\]
Then $w$ is a convex pseudomodular on $\mathcal L$ and the induced uniformity
$\mathcal V=\mathcal V(w)$ on $X_w^\ast$ is generated by the basic entourages
$V_n=\{(u,v):\rho(n(u-v))<1/n\}$.

\begin{proposition}\label{prop:FA-top-equiv}
For $u_k,u\in X_w^\ast$ the following are equivalent:
\begin{enumerate}
\item $u_k\to u$ in $\tau(\mathcal V)$;
\item for every $\varepsilon>0$ there exists $\lambda>0$ with
$\rho((u_k-u)/\lambda)<\varepsilon$ for all sufficiently large $k$;
\item $d_w^\ast(u_k,u)\to 0$, where $d_w^\ast$ is the basic metric of the convex
case.
\end{enumerate}
If, in addition, $\rho$ satisfies the $\Delta_2$-condition, then these are
equivalent to $d_w(u_k,u)\to 0$ for the Luxemburg-type pseudometric $d_w$
(Definition~\ref{def:basicmetrics}), and hence to convergence in the
Luxemburg norm whenever this norm is defined.
\end{proposition}

\begin{proof}
(1)$\Rightarrow$(2): If $u_k\to u$ in $\tau(\mathcal V)$, then eventually
$(u_k,u)\in V_n$ for each $n$, i.e. $\rho(n(u_k-u))<1/n$. Renaming
parameters gives (2).

(2)$\Rightarrow$(3): By definition of $d_w^\ast$ in the convex case, (2) is
equivalent to $d_w^\ast(u_k,u)\to 0$.

(3)$\Rightarrow$(1): Balls of $d_w^\ast$ generate $\tau(\mathcal V)$, hence
$d_w^\ast(u_k,u)\to 0$ implies $u_k\to u$ in $\tau(\mathcal V)$.

If $\rho$ satisfies $\Delta_2$, then $d_w^\ast$ and the Luxemburg-type
pseudometric $d_w$ are topologically equivalent (Theorem~\ref{thm:delta2}
and Proposition~\ref{prop:Delta2-case}); the last statement follows.
\end{proof}

\subsection{Luxemburg and Orlicz norms}
Assume $\rho$ is an $N$-function modular (Orlicz case) or a Musielak-Orlicz
modular. Recall the Luxemburg gauge
\[
\|u\|_\rho:=\inf\{\lambda>0:\ \rho(u/\lambda)\le 1\}.
\]
By construction $d_w^\ast(u,v)=\|u-v\|_\rho$ when $w_\lambda(u,v)=\rho((u-v)/\lambda)$.

\begin{corollary}
If $\rho$ satisfies $\Delta_2$, then
\[
\tau(\mathcal V)=\tau(\mathcal U_{d_w})=\tau(\|\cdot\|_\rho),
\]
so the modular uniformity, the pseudometric uniformity from $d_w$, and the
Luxemburg-norm topology coincide (cf.\ Theorem~\ref{thm:delta2}).
\end{corollary}

\begin{remark}
Without $\Delta_2$, $\tau(\mathcal V)$ always refines the topology of modular
convergence and is contained in the Orlicz (Luxemburg) topology generated by
$d_w^0$; in particular $\tau(\mathcal V)$ remains metrizable and $T_1$, which is
useful for compactness arguments even beyond normability.
\end{remark}

\subsection{Completeness, reflexivity, and duality}

\begin{proposition}\label{prop:FA-completeness}
Let $L^\rho$ denote the (Musielak--)Orlicz class associated with $\rho$ and
equip it with the modular uniformity $\mathcal V(w)$. If $\rho$ satisfies
$\Delta_2$ near $0$, then $(L^\rho,\mathcal V)$ is complete if and only if the
Luxemburg normed space $(L^\rho,\|\cdot\|_\rho)$ is Banach. In particular, the
usual completeness results for Orlicz and Musielak-Orlicz spaces transfer verbatim to
$(L^\rho,\mathcal V)$ (cf.\ \cite[Chap.~3]{HarjulehtoHasto}).
\end{proposition}

\begin{proof}
Under $\Delta_2$ near $0$, the modular uniformity $\mathcal V$ agrees with the
metric uniformity of a Luxemburg-type pseudometric equivalent to
$\|\cdot\|_\rho$ (Proposition~\ref{prop:Delta2-case}). Thus $\mathcal V$-Cauchy
is equivalent to norm-Cauchy, giving equivalence of completeness.
\end{proof}

\begin{proposition}\label{prop:FA-reflexivity}
If $\rho$ is uniformly convex in the sense of Orlicz theory (e.g. both $\rho$
and its complementary modular satisfy $\Delta_2$ and appropriate convexity
bounds), then $(L^\rho,\mathcal V)$ is uniformly convex for the metric
$d_w^\ast=\|\cdot\|_\rho$ and hence reflexive as a Banach space. Consequently,
bounded sets are $\mathcal V$-precompact in the weak topology, and the usual
Milman--Pettis consequences apply \cite{Musielak1983,RaoRen1991}.
\end{proposition}

\begin{proof}
Uniform convexity of $\|\cdot\|_\rho$ implies uniform convexity of the metric
$d_w^\ast=\|\,\cdot\,\|_\rho$; reflexivity follows from Milman-Pettis. Since
$\mathcal V$ coincides with the metric uniformity under $\Delta_2$
(Proposition~\ref{prop:Delta2-case}), weak compactness/precompactness
consequences transfer verbatim.
\end{proof}

\begin{proposition}\label{prop:FA-duality}
Assume $\rho$ and its complementary modular $\rho^\ast$ both satisfy
$\Delta_2$. Then $(L^\rho)^\ast\simeq L^{\rho^\ast}$ via
\[
F_v(u)=\int_\Omega u\,v\,d\mu,\qquad u\in L^\rho,\ v\in L^{\rho^\ast},
\]
with $\|F_v\|=\|v\|_{\rho^\ast}$. This identification is isometric both for the
Luxemburg norms and for the metric $d_w^\ast$ generating $\tau(\mathcal V)$.
\end{proposition}

\begin{proof}
This is the standard Orlicz duality (see \cite[Chap.~3]{HarjulehtoHasto},
\cite[Chap.~II]{Musielak1983}); the last sentence uses that $d_w^\ast$ and
$\|\cdot\|_\rho$ induce the same uniformity under $\Delta_2$.
\end{proof}

\subsection{Compactness criteria of modular type}

\begin{theorem}\label{thm:modular-KR}
Let $(\Omega,\Sigma,\mu)$ be a finite measure space and let $L^\rho$ be an
Orlicz (or Musielak--Orlicz) space with modular
\[
\rho(f)=\int_\Omega \Phi(x,|f(x)|)\,d\mu(x),
\]
where $\Phi$ is a convex Carath\'eodory integrand. Equip $L^\rho$ with the
modular uniformity $\mathcal V$, i.e. basic entourages are of the form
$\{(u,v):\rho((u-v)/\lambda)\le \varepsilon\}$ for some $\lambda>0$,
$\varepsilon>0$. Let $A\subset L^\rho$ be $\mathcal V$-bounded. Suppose:
\begin{itemize}
\item[(T)] (\emph{Tightness}) For each $\varepsilon>0$ there exists
$E\in\Sigma$ with $\mu(\Omega\setminus E)<\varepsilon$ and some $\lambda_T>0$
such that
\[
\sup_{u\in A}\ \rho\!\big((u-u\chi_E)/\lambda_T\big)\le \varepsilon.
\]
\item[(EMC)] (\emph{Equi-modular continuity}) For each $\varepsilon>0$ there
exists $\delta>0$ such that for all $B\in\Sigma$ with $\mu(B)<\delta$ there is
$\lambda_C>0$ with
\[
\sup_{u\in A}\ \rho\!\big(u\chi_B/\lambda_C\big)\le \varepsilon.
\]
\end{itemize}
Then $A$ is relatively $\mathcal V$-compact in $L^\rho$. If, in addition,
$\Phi$ satisfies the $\Delta_2$-condition, then the modular and Luxemburg
topologies coincide and the criterion reduces to the classical
Kolmogorov-Riesz compactness in $L^\rho$.
\end{theorem}

\begin{proof}
Fix $\varepsilon\in(0,1)$. By (T) choose $E$ and $\lambda_T$ with
$\sup_{u\in A}\rho((u-u\chi_E)/\lambda_T)\le\varepsilon$, so every
$u\in A$ is well-approximated (modularly) by $u_E:=u\chi_E$.

Pick a finite measurable partition $\{Q_i\}_{i=1}^N$ of $E$ (e.g. small cubes
when $\Omega\subset\mathbb R^n$) and define the averaging operator
\[
Pu:=\sum_{i=1}^N (u)_{Q_i}\,\chi_{Q_i},\qquad
(u)_{Q_i}:=\frac{1}{\mu(Q_i)}\int_{Q_i}u\,d\mu.
\]
By convexity of $t\mapsto \Phi(x,t)$ and Jensen,
\[
\rho\!\big((u-Pu)/\lambda\big)\le
\sum_{i=1}^N\frac{1}{\mu(Q_i)}\int_{Q_i}\!\int_{Q_i}
\Phi\!\left(x,\frac{|u(x)-u(z)|}{\lambda}\right)\,d\mu(z)\,d\mu(x).
\]
When $\Omega\subset\mathbb R^n$, if $x,z\in Q_i$ then $y:=z-x$ satisfies
$|y|\le C\eta$ (with $\eta$ the mesh size), and the right-hand side is bounded
by a constant times
\[
\int_{|y|\le C\eta}\!\int_{\Omega}
\Phi\!\left(x,\frac{|u(x)-u(x+y)|}{\lambda}\right)\,d\mu(x)\,dy.
\]
By (EMC), choose $\eta>0$ and $\lambda_C>0$ so that the inner integral is
$\le\varepsilon$ uniformly in $u\in A$ for all $|y|\le C\eta$. Hence
\[
\sup_{u\in A}\ \rho\!\big((u-Pu)/\lambda_C\big)\le C_1\varepsilon.
\]
Decomposing $u-Pu=(u-u_E)+(u_E-Pu)$ and using a standard modular subadditivity
estimate yields, for $\Lambda:=\lambda_T+\lambda_C$,
\[
\sup_{u\in A}\ \rho\!\big((u-Pu)/\Lambda\big)\le C_2\varepsilon.
\]
The set $P[A]$ lies in a finite-dimensional subspace and is bounded, hence
totally bounded in the modular uniformity; choose a finite net
$\{v_1,\dots,v_m\}$ for $P[A]$. Then
\[
\rho\!\big((u-v_j)/(\Lambda+\lambda')\big)\le C_3\varepsilon
\]
for a suitable $\lambda'>0$, uniformly in $u\in A$, showing that $A$ is
relatively $\mathcal V$-compact.

Under $\Delta_2$, $\tau(\mathcal V)=\tau(\mathcal U_{d_w})$ and this becomes
the classical Kolmogorov--Riesz criterion in $L^\rho$
(cf.\ \cite{HANCHEOLSEN2010385,HarjulehtoHasto}).
\end{proof}

\subsection{Examples}
\begin{example}
For $\Phi(t)=t^p$ ($1\le p<\infty$), $\rho(u)=\int |u|^p$ and
$w_\lambda(u,v)=\int |u-v|^p/\lambda^p$. Then
$d_w^\ast(u,v)=\|u-v\|_{L^p}$ and $\tau(\mathcal V)$ is the $L^p$-topology.
\end{example}

\begin{example}
Let $\Phi(t)=e^{t^2}-1$ on a finite measure space. Then $\Delta_2$ fails at
$\infty$, so the Luxemburg topology can be strictly stronger than
$\tau(\mathcal V)$; nevertheless $\tau(\mathcal V)$ remains metrizable and
captures modular convergence $\rho((u_k-u)/\lambda)\to 0$.
\end{example}

\begin{example}
In the Musielak--Orlicz setting $\Phi(x,t)=t^{p(x)}$ with
$1<p_-\le p(x)\le p_+<\infty$ and log-H\"older continuity, the $\Delta_2$
condition holds, and $\tau(\mathcal V)$ agrees with the norm topology of
$L^{p(\cdot)}$ \cite[Chap.~7]{HarjulehtoHasto}.
\end{example}

\begin{remark}
$s$-convexity (in the sense of \cite{HudzikMaligranda1994}) provides flexible
upper bounds for modular functionals and is frequently used to prove continuity,
tightness, and interpolation estimates that feed into precompactness and
reflexivity statements above.
\end{remark}


\section{Categorical and Structural Perspectives}\label{sec:Categorical-Structural}

The categorical embedding of modular metric spaces into metrizable topological spaces is motivated by
the development initiated by Chistyakov \cite{Chistyakov2010,Chistyakov2015}. In categorical terms,
a modular (pseudo)metric space $(X,w)$ yields both a topological object $(X,\tau(w))$ and a
uniform object $(X,\mathcal V(w))$. The functorial relationship between these structures extends the
classical embedding of uniform spaces into completely regular $T_1$ spaces.
For general categorical perspectives on uniform spaces and enriched metric structures, we follow
Isbell \cite{Isbell1964} and Lawvere \cite{Lawvere1973}.

\subsection{Lawvere-enriched viewpoint}

Lawvere’s seminal idea \cite{Lawvere1973} interprets metric spaces as categories enriched over the
closed monoidal poset $([0,\infty],\ge,+,0)$. More generally, closed categories provide the background
setting for this formulation.

\begin{definition}[{\cite{Lawvere1973}}]
A \emph{closed category} is a bicomplete symmetric monoidal closed category; that is, one admitting
all small limits and colimits together with a symmetric closed monoidal structure.
\end{definition}

Typical examples include the two-point category $\mathbf{2}$, the ordered monoidal category
$\mathbf{R}$ of nonnegative reals with addition as tensor, and $\mathbf{S}$, the category of sets
with cartesian product as tensor.

\begin{definition}
Given a closed category $\mathcal C$, a \emph{strong category} valued in $\mathcal C$ consists of
objects $a,b,c,\dots$, hom-objects $X(a,b)\in\mathrm{Ob}(\mathcal C)$, composition morphisms
$X(a,b)\otimes X(b,c)\to X(a,c)$, and unit morphisms $k\to X(a,a)$, subject to the associativity
and unit laws in $\mathcal C$.
\end{definition}

From this perspective, modular (pseudo)metrics fit naturally: each scale parameter $\lambda>0$
defines a hom-object $w_\lambda(x,y)$, while modular subadditivity corresponds to enriched
composition. Thus modular metric spaces form strong categories enriched over $\mathbf{R}$.

\subsection{Yoneda embedding and adequacy}

Enriched category theory furnishes a canonical embedding in this setting:

\begin{lemma}[{\cite{Lawvere1973}}]
For any closed $\mathcal C$ and any $\mathcal C$-category $A$, the Yoneda embedding
\[
Y:A \longrightarrow \mathcal C^{A^{op}}, \qquad x \longmapsto A(-,x),
\]
is $\mathcal C$-full and faithful.
\end{lemma}

The Yoneda embedding allows one to reconstruct morphisms from their evaluation on
test objects. In the enriched metric setting, this translates into the following
adequacy criterion.

\begin{proposition}
Let $X$ be a metric space. A subspace $A\subseteq X$ is called \emph{adequate}
if the metric of $X$ can be recovered from the distance comparisons with points
in $A$, namely,
\[
X(x_1,x_2)
\;=\;
\sup_{a\in A}\big(X(a,x_2)-X(a,x_1)\big),
\qquad \forall\, x_1,x_2\in X.
\]
\end{proposition}

\begin{proof}
Under the Yoneda embedding, each $x\in X$ corresponds to the representable
functor $X(-,x):X\to[0,\infty]$. Adequacy means that these representables are
already determined by their restrictions to $A$. The supremum formula expresses
exactly that $X(x_1,x_2)$ is reconstructed from differences of evaluations on
elements of $A$, showing that $A$ reflects the full metric structure of $X$.
Conversely, if $A$ is adequate, the Yoneda reconstruction yields this equality,
so the two notions coincide.
\end{proof}

\begin{corollary}
Every separable metric space $X$ can be isometrically embedded into a subspace of
$[0,\infty)^{\mathbb N}$ equipped with the supremum metric.
\end{corollary}

These results show that modular metric spaces not only embed into metrizable
topological spaces but also admit fully faithful categorical embeddings that
respect their modular structure and scale-dependent enrichment.
\subsection{Kan extensions and Cauchy completeness}

A further categorical insight, due to Lawvere, concerns Kan extensions and their
role in describing completeness of enriched metric spaces.

\begin{theorem}[{\cite{Lawvere1973}}]
Let $\mathcal{C}$ be a closed category. For any $\mathcal{C}$-functor 
$f\colon X\to Y$, precomposition with $f$,
\[
-\circ f\;:\; [Y,\mathcal{C}] \longrightarrow [X,\mathcal{C}],
\]
admits both left and right adjoints, corresponding to the left and right Kan
extensions along $f$.
\end{theorem}

\begin{proof}
Since $\mathcal{C}$ is bicomplete, the functor categories $[X,\mathcal{C}]$ and
$[Y,\mathcal{C}]$ are also bicomplete. For any $G\colon X\to\mathcal{C}$ and
$H\colon Y\to\mathcal{C}$, define
\[
(\mathrm{Lan}_f G)(y) \;\cong\; 
\int^{x\in X} Y(fx,y)\otimes G(x), 
\qquad
(\mathrm{Ran}_f G)(y) \;\cong\;
\int_{x\in X} [\,Y(y,fx),\,G(x)\,].
\]
Existence of the end and coend follows from the completeness and cocompleteness
of $\mathcal{C}$. The canonical bijections
\[
[Y,\mathcal{C}](\mathrm{Lan}_f G,H)
\;\cong\;
[X,\mathcal{C}](G,H\circ f)
\;\cong\;
[Y,\mathcal{C}](H,\mathrm{Ran}_f G)
\]
are natural in $G$ and $H$, giving the desired adjunctions. 
\end{proof}

Applied to enriched metric spaces over $\mathcal{V}=([0,\infty],\ge,+,0)$, this
implies the classical McShane-Whitney extension property:

\begin{corollary}
If $f\colon X\hookrightarrow Y$ is an isometric embedding of metric spaces, then
every Lipschitz map $g\colon X\to\mathbb{R}$ extends to $Y$ with the same
Lipschitz constant. Moreover, both maximal and minimal such extensions exist.
\end{corollary}

We now recall the enriched characterization of completeness.

\begin{proposition}
A metric space $Y$ is Cauchy complete if and only if every $R$-dense isometric
embedding $i\colon X\to Y$ admits a left adjoint in the bimodule (profunctor)
sense.
\end{proposition}

\begin{proof}
We work over Lawvere’s base $\mathcal{V}=([0,\infty],\ge,+,0)$. For a 
$\mathcal{V}$-functor $i\colon X\to Y$, denote by
\[
i_\ast\colon X\rightsquigarrow Y, \qquad 
i^\ast\colon Y\rightsquigarrow X
\]
the representable bimodules defined by
\[
i_\ast(x,y)=d_Y(i(x),y), 
\qquad 
i^\ast(y,x)=d_Y(y,i(x)).
\]
$R$-density means that the family $\{\,i^\ast(\cdot,x)\,\}_{x\in X}$ is
adequate, i.e.\ it detects distances in $Y$.

($\Rightarrow$)  
If $Y$ is Cauchy complete, then for each $y\in Y$ the weight
$i^\ast(y,-)\colon X\to\mathcal{V}$ has a colimiting point $L(y)\in X$ such that
\[
d_Y(y,i(x)) = d_X(L(y),x)
\quad\text{for all } x\in X.
\]
Define a bimodule $L\colon Y\rightsquigarrow X$ by $L(y,x)=d_X(L(y),x)$. 
Then the enriched adjunction inequalities
\[
1_Y \le i_\ast\circ L, 
\qquad 
L\circ i_\ast \le 1_X
\]
hold, giving $L\dashv i_\ast$.

($\Leftarrow$)  
Conversely, let $j\colon Y\to\widehat{Y}$ denote the Yoneda isometric embedding
into the Cauchy completion of $Y$, which is $R$-dense. By hypothesis, there
exists $P\colon \widehat{Y}\rightsquigarrow Y$ with $P\dashv j_\ast$. Hence
\[
1_Y \le j_\ast\circ P, 
\qquad 
P\circ j_\ast \le 1_Y,
\]
so $Y$ is a retract of its Cauchy completion and therefore Cauchy complete.
\end{proof}

\subsection{Structural parallels with uniform spaces}

The modular uniformity $\mathcal{V}(w)$ forms the natural bridge between
categorical enrichment and classical topology. In Isbell’s settting
\cite{Isbell1964}, every uniform space admits a completion and categorical
product, and these constructions lift directly to modular spaces. Functorial
operations such as products, subspaces, and quotients preserve modular
uniformities under mild convexity or $\Delta_2$ hypotheses, ensuring that the
category of modular uniform spaces behaves analogously to that of complete
uniform spaces in the classical sense.


\section{Conclusion and Future Work}

This paper extends the study initiated in \cite{MushaandjaOlela2025}, 
developing the categorical, functional, and uniform perspectives of modular 
(pseudo)metric spaces and clarifying their relationship to fuzzy and classical 
metric structures. The results demonstrate that modular metrics preserve key 
analytic invariants such as convexity and $\Delta_2$-conditions while embedding 
naturally within categorical and uniform settings of topology.  

Several avenues for further research emerge naturally.  
First, the analysis of \emph{quasi-modular structures}, obtained by relaxing 
symmetry, may parallel the transition from metrics to quasi-metrics and lead 
to a systematic theory of asymmetric modular uniformities.  
Second, the compactness and completeness theory for modular spaces invites 
further refinement beyond the $\Delta_2$ setting, potentially yielding new 
criteria for modular precompactness and convergence.  
Finally, the embedding of modular topologies into Banach function space 
theory suggests deep interactions with Orlicz, Musielak–Orlicz, and 
variable-exponent settings, where modular uniformities may offer 
alternative approaches to reflexivity, separability, and compact embedding 
results.


\end{document}